\documentclass[11pt]{amsart}

\usepackage{amsmath,amsthm,amssymb,amsfonts, amsopn, mathtools}
\usepackage{stmaryrd,tikz,tikz-cd}
 \usepackage{todonotes}
\usepackage{tcolorbox}
\usepackage[colorlinks=true, linktocpage=true]{hyperref}
\hypersetup{citecolor=blue}
\usepackage[all]{xy}
\usepackage{MnSymbol}

\usepackage{stackengine}
\def\subrangle#1{\stackengine{5pt}{}{$\!\scriptstyle #1$}{U}{l}{F}{F}{L}}
\newcommand{\norm}[1]{\left\lVert#1\right\rVert}
\newcommand{\End}{\textnormal{End}}

\newtheorem{theorem}{Theorem}

\newtheorem*{lemma*}{Lemma}         
\newtheorem{proposition}[theorem]{Proposition}
\newtheorem*{proposition*}{Proposition}

\theoremstyle{definition}

\newcommand{\C}{\mathbb C}

\newcommand{\R}{\mathbb R}

\title{Theta correspondence as a $C^*$-correspondence}

\author[\c{S}eng\"un]{Mehmet Haluk \c{S}eng\"un}
\address{\normalfont{School of Mathematical and Physical Sciences\\
University of Sheffield,
Hounsfield Road, Sheffield, S3 7RH, UK}}
\email{m.sengun@sheffield.ac.uk}

\begin{document}

\begin{abstract}
This is a survey of the recent works that describe theta correspondence as a $C^*$-correspondence. The main mechanism is elucidated in the special case when one of the groups involved in the theta correspondence is compact. 
\end{abstract}

 \maketitle
\tableofcontents


\section{Introduction}
This is a survey of the recent works \cite{Mesland-Sengun, Goffeng-Mesland-Sengun} that interpret Jian-Shu Li's explicit approach (\cite{Li-89}) to theta lifting (\cite{Howe-79})  through induction theory (\cite{Rieffel-74}) of representations of $C^*$-algebras.  I will focus on the local case here, however the approach works equally well in the global case \cite{GMS-adelic, Goffeng-Mesland-Sengun}. The exposition is catered towards readers with background in representation theory, in particular, the survey starts with an informal overview the required $C^*$-algebra background, focusing on group $C^*$-algebras and induction theory. While the background discussions have unavoidable overlap with those in \cite{Mesland-Sengun, Goffeng-Mesland-Sengun}, they are different in content and presentation. The reader is recommended to consult \cite{Mesland-Sengun-book-chapter, Crisp} for related materials.

Dual groups in which one member is compact is fundamental to the local theta correspondence (LTC); this is the best understood case of the correspondence and it is crucially used to understand the general case.  The survey is built around this ``{\em compact case}'' of the LTC as in this case much of the technical challenges are absent, making the main constructions that I would like to expose more transparent.

The survey is organised as follows. In Section \ref{cstarback}, I discuss the background in $C^*$-algebras, correspondences/Hilbert $C^*$-modules, and in Section \ref{cstarcomp}, I zoom in on the case of the $C^*$-algebra of a compact group. I continue by recalling the setup for the local theta correspondence for a dual pair in Section \ref{theta-background}, with the succeeding Section \ref{theta-compact-case} discussing the case of compact $G$ and Section \ref{sec:liswork} recalling the salient points in the work of Jian-Shu Li \cite{Li-89} needed for our discussion. Section \ref{compaclatatd} gives full details for how Rieffel induction as in \cite{Mesland-Sengun, Goffeng-Mesland-Sengun} recovers the theta correspondence in the case of compact $G$. In the last two sections, I survey the general results of \cite{Mesland-Sengun, Goffeng-Mesland-Sengun} with Section \ref{bimodule-construction} and Section  \ref{seceualada} describing the further $C*$-algebraic structures present in the equal rank case.

\subsection{Acknowledgements}
First and foremost, I thank my collaborators Magnus Goffeng and Bram Mesland. It was a great pleasure to explore these exciting connections together. This survey improved greatly with Magnus's insightful comments and corrections. Over the years, I benefitted from correspondences and conversations with many colleagues, too many to name here (but named in \cite{Mesland-Sengun, Goffeng-Mesland-Sengun}). I thank them all warmly. I am grateful to Roger Howe, Wee Teck Gan and Tomasz Przebinda for numerous mathematical conversations which I benefitted greatly from, and also for their interest and encouragement. Most of the major advances in this project were obtained during a fellowship funded by the EPSRC grant EP/V049119/1. I also wish to thank the Institut Henri-Poincaré for their hospitality during the thematic programme ``\emph{Representation Theory \& Noncommutative Geometry''} during the spring 2025. Finally, I thank the referee for their feedback.

\section{$C^*$-algebra background}
\label{cstarback}

Let us quickly review the basic $C^*$-algebra theory that we will need. We will borrow some parts from \cite[Section 2]{Goffeng-Mesland-Sengun}. We also refer the reader to \cite{Crisp} in this proceedings, as well as text books in $C^*$-algebras \cite{Dixmier,strung} and Hilbert $C^*$-modules \cite{Lance,raeburnwill}. Throughout the paper, \emph{we follow the Hilbert $C^*$-module convention where inner products on Hilbert spaces are linear in the second variable. }

\subsection{$C^*$-algebras} Recall that a {\bf $C^*$-algebra} is an involutive (complex) Banach algebra $A$ satisfying the so-called {\em $C^*$-identity}:
$$\norm{aa^*}=\norm{a}^2$$
for all  $a \in A$. 

The endomorphism algebra $\End(V)$ will denote the $C^*$-algebra of bounded linear operators on a Hilbert space $V$, with its operator norm and adjoint involution ${}^*$. It is the most fundamental example of a $C^*$-algebra. A closed $*$-subalgebras of $\End(V)$, for some $V$, is called a concrete $C^*$-algebra. A {\bf representation of a $C^*$-algebra} $A$ is a homomorphism $\pi : A \to \End(V)$ of $C^*$-algebras where $V$ is a Hilbert space. A foundational result of Gelfand-Naimark-Segal says that every $C^*$-algebra admits an injective  representation so that all $C^*$-algebras are isomorphic to concrete $C^*$-algebras.

Given two representations $\pi_1,\pi_2$ of $A$, we say that $\pi_1$ is {\bf weakly contained} in $\pi_2$ is the kernel of $\pi_1$ {\em contains} the kernel of $\pi_2$. Note that the kernel of a representation is a two-sided closed ideal of $A$ and such an ideal is, in particular, a $C^*$-subalgebra.

\subsection{Topology} The dual (i.e. the set of (equivalence classes of) irreducible representations) $\widehat{A}$ of a $C^*$-algebra $A$ carries a topology in which the open set are of given by duals of closed, two-sided ideals of $A$. Indeed such an ideal $J$ is actually a $C^*$-subalgebra of $A$ and its dual consists of those irreps of $A$ which do not vanish on $J$: 
$$\widehat{J}=\{ \pi \in \widehat{A} \mid \pi(J)\not= 0 \}.$$
The complementary closed subsets then are of the form 
$$\widehat{A}\setminus \widehat{J} = \{ \pi \in \widehat{A} \mid \pi(J) = 0 \}$$
for some closed, two-sided ideal $J$ of $A$. We point out that clopen (closed and open) subsets of $\widehat{A}$ correspond to complemented two-sided ideals of $A$.  

More generally, the collection\footnote{There are some set-theoretic complications here. These can be avoided by restricting to representations on separable Hilbert spaces.} ${\rm Rep}(A)$ of (equivalence classes of) all representations of $A$ can be equipped with a topology extending the one on $\widehat{A}$ described above. This topology is intimately related to the notion of weak containment mentioned earlier; the closure of a set $S$ of representations is given by the collection of representations which are weakly contained in each of the representations of $S$.

\subsection{Hilbert $C^*$-modules}\label{Hilbert-module}
There is an intrinsic notion of positivity in $C^*$-algebras; $a \in A$ is {\bf positive} if it is equal to $bb^*$ for some $b \in A$. Given an injective homomorphism of $A$ into some  $\End(V)$, positivity of $a$ is equivalent to $a$ acting as a positive operator on $V$ in the sense of quadratic forms\footnote{That is, the quadratic form $q_a(x):=\langle x,ax\rangle$ is positive in the sense that $q_a(x)\geq 0$ for any $x\in V$.}.  Using this notion of positivity, we can talk about a module over $A$ which is equipped with a positive-definite $A$-valued inner product. Completing such a module with respect to the norm arising from the $A$-valued inner product give us the notion of a Hilbert module over $A$, generalizing the notion of a Hilbert space (where $A=\C$).

 A {\bf Hilbert $C^{*}$-module} over a $C^*$-algebra $B$ is a complex vector space $X$ that is a right $B$-module equipped with a map $\langle {\cdot}, {\cdot} \rangle : X {\times} X \to B$ such that for $x,y,z\in X$, $b \in B$ and $\lambda,\mu\in \C$, we have
\begin{enumerate}
\item[(i)] $\langle x, \lambda y + \mu z \rangle = \lambda \langle x, y \rangle + \mu \langle x, z \rangle$,
\item [(ii)] $\langle x,y \rangle^* = \langle y,x \rangle$, 
\item[(iii)] $\langle x, y{\cdot}b \rangle = \langle x, y \rangle b$,
\item[(v)] $\langle x,x \rangle \geq 0$,
\item[(vi)] $\langle x,x \rangle =0$ if and only if $x=0$,
\item[(vi)] $X$ is complete with respect to the norm $||x||:=||\langle x,x \rangle ||_{B}^{1/2}$. 
\end{enumerate}
We say that $X$ is {\em full} if $\langle X,X \rangle$ spans a dense subspace in $B$. Note that, as is standard in $C^*$-theory, the inner products above are {\em right} linear. This is compatible with our over all convention on Hilbert spaces -- the Hilbert modules arising for $B=\C$.

Now let $X$ be a Hilbert $C^{*}$-module over $B$. We let $\End_B ^*(X)$ denote the algebra of $B$-module endomorphisms $T:X\to X$ that are \emph{adjointable} in the sense that there exists $T^{*}:X\to X$ such that $\langle Tx,y\rangle=\langle x,T^{*}y\rangle$.  The vector space $\End_B^*(X)$ is a subspace of the space of bounded $B$-module endomorphisms and is a $C^*$-algebra under the operator norm that one derives from that on $X$. In the special case $B=\C$, a $B$-Hilbert $C^*$-module is a Hilbert space and in this case $\End_\C^*(V)\equiv \End(V)$.

Given $x,y \in X$, we define the operator $T_{x,y}$ by $T_{x,y}(z)=x \langle y,z\rangle$ for every $z \in X$. The linear span of the rank one operators form a two-sided ideal of $\End_B ^*(X)$ and the norm closure of this ideal is a $C^{*}$-algebra known as the ideal of {\em $B$-compact operators}, denoted $\mathbb{K}_B(X)$. We point out that these operators are not necessarily compact in the sense of Banach space theory.

The linear span of the image $\langle X,X \rangle$ in $B$ is a two-sided ideal of $B$. We will refer to the {\em closure} of this ideal as the {\bf ideal cut out by the inner product} and will denote it $J_B$. Note that $X$ can be viewed as a full Hilbert module over $J_B$.

\subsection{Example} 
\label{first-example} 
Let $(V,\langle{\cdot},{\cdot}\rangle)$ be a Hilbert space with endomorphism algebra $E:=\End(V)$. We write $V^*:=\mathrm{Hom}(V,\C)$ for the continuous linear maps $V\to \C$. View $V^*$ as a right module over $E$ via
$$v{\cdot}T := v\circ T, \quad v \in V^*, \ \ T \in E$$
and equip it with the $E$-valued form 
\begin{equation} 
\label{compact-inner-product} 
\llangle w,v \rrangle:= T_{w,v}, \quad v,w\in V^*,
\end{equation} 
where $T_{w,v}$ is the rank-one operator $T_{w,v}(z)=w^*v(z)$ where $w^*\in V$ is the element uniqely determined by $w(z)=\langle w^*,z\rangle$. Note that $w\mapsto w^*$ defines an antilinear isomorphism $V^*\to V$. The expression \eqref{compact-inner-product} defines a right-linear Hermitian form that is compatible with the right $E$-module structure of $V$; that is, it satisfies properties (i),(ii) and (iii) of Section \ref{Hilbert-module}. 
The intrinsic definition of an endomorphism $T$ being positive is equal to saying that $T$ acts as a positive operator on $V$, which can be characterised as having $\langle T(v), v \rangle \geq 0$ for all $v \in V^*$. It follows immediately from the identity
$$\langle \llangle v,v \rrangle(z),z \rangle = \langle v^* v(z),z\rangle =  \left | v(z)\right |^2, \quad z\in V, \, v\in V^*,$$ 
that the form $\llangle {\cdot},{\cdot} \rrangle$ is positive-definite; it satisfies properties (iv) and (v) of Section \ref{Hilbert-module}.

Finally, given $v \in V^*$, introduce the new norm $\norm{v}:=\norm{\llangle v,v\rrangle}_E^{1/2}$
where $\norm{{\cdot}}_E$ is the operator norm on $E$. Observe that the operator norm of the rank one operator $\llangle v,v\rrangle$ is simply the square of the usual Hilbert space norm of $v$. Therefore the new norm is simply equal to the Hilbert space norm. We conclude that $V^*$ is already complete with the new norm and hence satifies property (vi) of Section \ref{Hilbert-module}. 

We have shown that $V^*$, with the above $E$-module structure and $E$-valued inner product, is a Hilbert module over $E$. Observe that the image of $\llangle {\cdot},{\cdot} \rrangle$ in $E$ is the ideal of rank one operators.  Hence the ideal cut out by the inner product is the ideal $\mathbb{K}(V)$ of compact operators on $V$. Thus $V^*$ is always a full Hilbert module over  $\mathbb{K}(V)$, and it is is a full Hilbert module over $E$ if and only if $V$ is finite dimensional. In summary, for a Hilbert space $V$, the space $V^*$ is naturally a right $\End(V)$-Hilbert $C^*$-module. We remark that $V$ would naturally be a \emph{left} $\End(V)$-Hilbert $C^*$-module with inner product defined from ${}_E\llangle w,v \rrangle(z):= w\langle v,z\rangle_V$ for $z,v,w\in V$.

\subsection{$C^*$-correspondences} Let $A,B$ be $C^*$-algebras and $X$ a Hilbert $C^*$-module over $B$. Assume that there is a homomorphism $\alpha: A \to \End_B ^*(X)$ of $C^*$-algebras. In such a situation, $X$ is called a {\bf $C^*$-correspondence} for $(A,B)$, or simply an {\em $(A,B)$-correspondence}. 

A {\bf representation of a $C^*$-algebra} $A$ is simply an $(A,\C)$-correspondence, in other words, it is a homomorphism 
$\alpha: A \to \End(V)$ where $V$ is a Hilbert space. We call $\alpha$ {\em non-degenerate} if $\alpha(A)V$ is dense in $V$.

\subsection{Induction} \label{induction} Given an $(A,B)$-correspondence $\alpha: A \to \End_B ^*(X)$, we can induce non-degenerate representations of $B$ to those of $A$ as follows. Let $\pi:B\to \End(V)$ be a representation of $B$ on a Hilbert space $V$. Consider the algebraic tensor product of vector spaces $X \otimes^{\mathrm{alg}} V$. The right sesquilinear form 
\begin{equation} \label{localisation-inner-product} \langle x \otimes v, x' \otimes v'\rangle := \langle v, \pi(\langle x,x' \rangle)v' \rangle\subrangle{V},
\end{equation}
can be shown to be positive. Therefore it decends to a positive-definite form on $(X \otimes^{\mathrm{alg}} V)/N$
where 
$$N:=\left\{\xi \in X \otimes^{\mathrm{alg}} V \mid \langle \xi, \xi \rangle = 0 \right\}$$
is the subspace of null vectors. Completion of $(X \otimes^{\mathrm{alg}} V)/N$ with respect to the inner product \eqref{localisation-inner-product} is a Hilbert space that we denote by $X \otimes_{B}V$ and is commonly called the {\em internal tensor product} of $X$ and $V$ over $B$.

Consider the action 
$$a(x \otimes v) := \alpha(a)(x) \otimes v$$ 
of $A$ on the space $X\otimes^{\mathrm{alg}} V$. The nontrivial fact, which actually will be important to us later in regards to theta correspondence, that the null space $N$ equals the \emph{balancing subspace}
\begin{equation}
\label{balancing-ideal}
I:=\mathrm{span}\left\{xb \otimes v - x\otimes \pi(b)(v):x \in X, \ v \in V, \ b \in B\right\}.
\end{equation} 
allows us to easily see that th action of $A$ descends down to the quotient$(X \otimes^{\mathrm{alg}} V)/N$. Standard arguments then can be used to show that this action extends to a representation of $A$ on the Hilbert space $X \otimes_B V$ which we will denote $\textnormal{\small Ind}_{B}^{A}(X, \pi),$
and refer to as the $A$-representation {\bf induced from $\pi$ via $X$}. This procedure gives us a functor
$${\rm Ind}_B^A(X) : {\rm Rep}(B) \to {\rm Rep}(A)$$
from the category of non-degenerate representations of $B$ to that of $A$. An important feature of this functor is that it is continuous with respect to {\em weak containment} of representations, that is, it respects the inclusion of kernels of representations.

\subsection{Morita equivalence} We say that two $C^*$-algebras $A$ and $B$ are {\bf Morita equivalent} if there exists an $(A,B)$-correspondence  $\alpha: A \to \End_B ^*(X)$ such that $X$ is a full Hilbert module over $B$ and the $C^*$-algebra $\mathbb{K}_B(X)$ of $B$-compact operators on $X$ is isomorphic to the image $\alpha(A)$. 

Morita equivalent $C^*$-algebras share many important features. In particular, if $A$ and $B$ are Morita equivalent, the functor ${\rm Ind}_B^A(X)$ establishes an equivalence between the categories of non-degenerate representations of $A$ and $B$. We also note that Morita equivalence is an equivalence relation.

Sometimes an $(A,B)$-correspondence $X$ comes with a natural {\em left}\footnote{For the definition of a left Hilbert module, we modify the definition of a (right) Hilbert module so that the module action is on the left and the inner product is linear in the left variable.} Hilbert $C^*$-module structure over $A$. If the two $C^*$-valued inner products are compatible in the following way
\begin{equation} \label{imprimitivity-compatibility}   \prescript{}{A}{\langle x, y \rangle} {\cdot} z = x {\cdot} \langle y,z \rangle\subrangle{B},
\end{equation}
then we call $X$ an $(A,B)$-{\bf imprimitivity bimodule}\footnote{In the literature, it is often required that $X$ is a full Hilbert module both over $A$ and over $B$. We do not insist on this.}. Here $x,y,z$ are in $X$ and $\prescript{}{A}{\langle {\cdot},{\cdot} \rangle}$,$\langle {\cdot},{\cdot} \rangle\subrangle{B}$ are the $A$-valued and $B$-valued inner products on $X$ respectively. The key point is that \eqref{imprimitivity-compatibility} implies that the ideals 
$J_A \subseteq A$ and $J_B \subseteq B$ cut out by the inner-products are Morita equivalent.

\subsection{Group $C^*$-algebras} Let $G$ be a locally compact Hausdorff topological group. Consider the convolution algebra $L^1(G)$ of integrable complex valued functions on $G$. By norming $L^1(G)$ with the $L^1$-norm, $L^1(G)$ forms an involutive Banach algebra for the involution 
$$f^*(s):=\Delta(s^{-1})\overline{f(s^{-1})}, \quad f \in L^1(G), \ s \in G,$$
where $\Delta$ is the modular functions of $G$. Just as in the case of a finite group and its group algebra, we can ``integrate'' a unitary representation $\pi:G \to \mathcal{U}(V)$ of $G$ to a $*$-homomorphism
$$\pi: L^1(G) \to \End(V), \quad \pi(f)=\int_G f(s) \pi(s)\ ds$$
which we call a representation of $L^1(G)$ on the Hilbert space $V$. 
This process leads to a bijection between unitary representations of $G$ and nondegenerate representations of $L^1(G)$. 

It was realized in the late 1940's by Irving Segal that it is better to promote $L^1(G)$ to a $C^*$-algebra\footnote{The $L^1$-norm does not satisfy the $C^*$-identity.} as follows; given $f \in L^1(G)$, introduce 
\begin{equation} \label{Cstar-norm} \norm{f}:= \sup_{\pi \in \widehat{G}} \norm{\pi(f)}_{\End(V_\pi)}.
\end{equation}
Here $\widehat{G}$ is the {\bf unitary dual} of $G$; the set of (equivalence classes of) irreducible unitary representation of $G$. It turns out that this formula defines a norm on $L^1(G)$ which satisfies the $C^*$-identity. We define the {\bf $C^*$-algebra of $G$}, denoted $C^*(G)$, as the completion of $L^1(G)$ with respect to the norm defined in \eqref{Cstar-norm}. For further motivation to study $C^*(G)$, see \cite{Crisp,Dixmier,strung}.

Representations of $L^1(G)$ extend uniquely to $C^*(G)$ and we obtain a bijection between unitary representations of $G$ and nondegenerate representations of $C^*(G)$. 
The unitary dual $\widehat{G}$ of $G$ carries the so-called Fell topology which is based on convergence of matrix coefficients. For example, in this topology, $\widehat{G}$ is discrete if $G$ is compact. The bijection between the duals of $G$ and $C^*(G)$ is actually a homeomorphism when we consider the topologies on the two sides. 

The {\bf reduced $C^*$-algebra of $G$}, denoted $C^*_r(G)$, is defined as the image of $C^*(G)$ under the (integrated form of the) regular representation of $G$ on $L^2(G)$. Thus $C^*_r(G)$ is a quotient of $C^*(G)$ and its representations correspond, under the above mentioned bijection, to those unitary representations of $G$ which are weakly contained in the regular representation. Such representations form a closed subspace of the unitary dual of $G$ that is called the {\bf reduced dual of $G$}. If $G$ is a reductive algebraic group, then the reduced dual is commonly known as the {\em tempered dual} since it coincides with the set of tempered unitary representations\footnote{That is, irreducible unitary representations whose matrix coefficients are $L^{2+\epsilon}$ for any $\epsilon>0$, see \cite{cowhaaghow}.}.

\section{$C^*$-algebras of compact groups} 
\label{cstarcomp}

As I wish to put the compact case of the theta correspondence at the heart of this survey, let us look at the case of a compact group $G$ in detail. The representation theory of a compact group is well understood, and further context can be found in for instance \cite{deitmarechterhoof} or \cite[Chapter 0 and 1]{wallach}.

By the Peter-Weyl Theorem, every irreducible unitary representation of $G$ is a subrepresentation of the regular representation so that the unitary dual and the reduced dual are the same. Moreover, as the topology on the unitary dual is discrete, every singleton is clopen. Each such clopen singleton corresponds to a complemented two-sided closed ideal as we discussed above. Therefore, $C^*(G)$ ``decomposes into'' a direct sum of these complemented ideals, one for each irrep of $G$. Let us be more precise.

\subsection{Isolated points and imprimitivity bimodules} Given an irreducible unitary representation $(\pi,V_\pi)$ of $G$, consider the integrated representation 
$$\pi: C^*(G) \to \End(V_\pi),$$
which is nothing but a $(C^*(G), \C)$-correspondence as we mentioned earlier. A very elegant way of studying the complemented two-sided ideal of $C^*(G)$ that corresponds to the clopen $\{ \pi \}$ is via promoting this correspondence to an imprimitivity bimodule. We do this as follows. 

Given $v,w \in V_\pi$,  introduce the matrix coefficient function $\prescript{}{G}\langle v,w \rangle$:
\begin{equation} 
\label{left-inner-compact-group} 
\prescript{}{G}\langle v,w \rangle(s):=d_\pi\langle \pi(s)(w), v \rangle\subrangle{V_\pi}, \quad v,w \in V_\pi, \ s \in G.
\end{equation}
Here $d_\pi$ is the dimension of $V_\pi$. As $G$ is compact, the matrix coefficients $\prescript{}{G}\langle v,w \rangle$ lie in $C^*(G)$ and $\prescript{}{G}\langle {\cdot},{\cdot} \rangle$ is valued in $C^*(G)$. It is easy to see that $\prescript{}{G}\langle {\cdot},{\cdot} \rangle$ is left linear, Hermitian and is compatible with the left $C^*(G)$-module structure on $V_\pi$ (thus, satisfies the first three properties of a left Hilbert module). 

We need $\prescript{}{G}\langle {\cdot},{\cdot} \rangle$ to be positive-definite. This will follow from the Schur orthogonality relations. Given $v \in V_\pi$, positivity of $\prescript{}{G}\langle v,v \rangle$ as an element of $C^*(G)$ is equivalent to $\sigma(\prescript{}{G}\langle v,v \rangle)$ being a positive operator on $V_\sigma$ for every irreducible unitary representation 
$(\sigma, V_\sigma)$ of $G$. Given such $(\sigma, V_\sigma)$, the operator $\sigma(\prescript{}{G}\langle v,v \rangle)$ is positive if for all $w \in V_\sigma$, we have 
$$\langle \sigma(\prescript{}{G}\langle v,v \rangle)(w),w\rangle\subrangle{V_\sigma} \geq 0.$$
This is true because
\begin{align*} \langle \sigma(\prescript{}{G}\langle v,v \rangle)(w),w\rangle\subrangle{V_\sigma} &= 
d_\pi\int_G \overline{\langle \pi(s)(v), v \rangle}\subrangle{V_\pi} \langle \sigma(s)(w),w\rangle\subrangle{V_\sigma} ds \\
&=\begin{cases}   |\langle w,v \rangle\subrangle{V_\pi} |^2  &{\rm if} \ \pi \simeq \sigma, \\  0 &{\rm else.}
\end{cases}
\end{align*}
The last equality follows from Schur's orthogonality relation of course.

We next turn to the compatibility condition \eqref{imprimitivity-compatibility}, which is the reason we introduced the factor $d_\pi$ in \eqref{left-inner-compact-group}. Indeed, we have 
\begin{align*} \langle \prescript{}{G}\langle x,y \rangle{\cdot}z, u \rangle\subrangle{V_\pi} &= 
\langle \pi(\prescript{}{G}\langle x,y \rangle)(z), u \rangle\subrangle{V_\pi} \\
&= d_\pi \int_G \overline{\langle \pi(s)(y), x \rangle}\subrangle{V_\pi} \langle \pi(s)(z), u \rangle_{V_\pi} \ ds\\ 
&= d_\pi \tfrac{1}{d_\pi}\langle z, y \rangle\subrangle{V_\pi} \overline{\langle u,x \rangle}\subrangle{V_\pi} \\
&=  \langle x \langle y, z \rangle\subrangle{V_\pi}, u \rangle\subrangle{V_\pi}
\end{align*}
for every $x,y,z,u \in V_\pi$. Therefore
$$ \prescript{}{G}\langle x,y \rangle{\cdot}z= x \langle y, z \rangle\subrangle{V_\pi}$$
as required in  \eqref{imprimitivity-compatibility}.

We have shown that $V_\pi$ has the structure of a $(C^*(G),\C)$-imprimitivity bimodule. Thus the two-sided closed ideal $J_\pi$ cut out by the inner product in \eqref{left-inner-compact-group} is isomorphic to the algebra $\mathbb{K}(V_\pi)$ of compact operators on the Hilbert space $V_\pi$. As such, its dual is the singleton $\{ \pi \}$:
$$\widehat{J_\pi}=\{ \pi \}.$$  
As $V$ is finite dimensional, we have $\mathbb{K}(V_\pi) = \End(V_\pi)$ so that each $J_\pi$ is isomorphic to $\End(V_\pi)$ and the representation $\pi$ of $C^*(G)$ is surjective. Thus $J_\pi$ is complemented:
$$C^*(G) \simeq J_\pi \oplus {\rm Ker}(\pi).$$

\subsection{The decomposition} Consider the $C^*$-subalgebra $\bigoplus _{\pi \in \widehat{G}} J_\pi\subseteq C^*(G)$. Here the direct sum is in the category of $C^*$-algebras, i.e. $\bigoplus _{\pi \in \widehat{G}} J_\pi$ consists of collections $(a_\pi)_{\pi\in \widehat{G}}\in \prod_{\pi \in \widehat{G}} J_\pi$ such that $(\|a_\pi\|)_{\pi\in \widehat{G}}\in C_0( \widehat{G})$. Equivalently, $\bigoplus_{\pi \in \widehat{G}} J_\pi$ is the closure of  the algebraic direct sum of the ideals $J_\pi$. 
As $C^*(G)$ is type I, we can conclude that
\begin{equation} 
\label{decomp-compact-Cstar}
C^*(G) = \bigoplus _{\pi \in \widehat{G}} J_\pi,
\end{equation}
from the simple facts that the spectrum of $\bigoplus _{\pi \in \widehat{G}} J_\pi$ coincides with $\widehat{G}$ and $\pi(C^*(G))=\End(V_\pi)=\pi(J_\pi)$ for all $\pi\in \widehat{G}$. Recalling that each $J_\pi$ is isomorphic to $\End(V_\pi)$, the above ``internal'' decomposition leads to the ``external'' one
\begin{equation} 
\label{decomp-compact-Cstar-2} 
C^*(G) \simeq \bigoplus _{\pi \in \widehat{G}} \End(V_\pi).
\end{equation}
The decompositions \eqref{decomp-compact-Cstar} and \eqref{decomp-compact-Cstar-2} are the $C^*$-algebraic reincarnations of the Peter-Weyl theorem.

\section{Local theta correspondence (LTC)} 
\label{theta-background}

The local theta correspondence describes a correspondence between representations of two reductive group $G$ and $G'$ fitting into a dual pair, i.e. $G$ and $G'$ are realized as subgroups of a symplectic group that are each others' centralizers \cite{Gelbart,Howe-79}. To simplify our discussion of the local theta correspondence, we will only work with {\em ortho-symplectic dual pairs}. This will not lead to any loss of generality from a conceptual point of view. 

Let $F$ be a local field. We fix even\footnote{We are assuming that the quadratic space is even dimensional in order to technically simplify our exposition by avoiding the need for covering groups.} dimensional vector spaces $V',V$ over $F$, one equipped with a skew-symmetric bilinear form and the other with a symmetric bilinear form. We consider $G'$ and $G$ that form an ortho-symplectic dual pair, that is they will be groups of isometries of $V',V$ respectively, so that one is the symplectic group and the other is the orthogonal group. Associated to a non-trivial unitary character $\chi$ of $F$, there is a well-known unitary representation $\omega=\omega_\chi$ of $G'{\times} G$, called the {\bf oscillator representation}. 

The oscillator representation is known to have many different models, each with different attractive features. One popular model is the so-called {\em Schr\"odinger model}. Let us assume that we have a decomposition $V'=V'_0 \oplus V'_1$ of $V'$ into two totally isotropic subspaces of equal dimensions, this decomposition polarizes $V'$ allowing for the construction of the Schr\"odinger representation of the relevant Heisenberg group. Then $G$ acts on the vector space 
$$U:={\rm Hom}(V'_0,V)$$
by post-composition. We can realise the oscillator representation on $L^2(U)$, 
$$\omega : G' {\times} G \longrightarrow \mathcal{U}(L^2(U)),$$
such that the action of $G$ is {\em geometric}, that is, 
$$\omega(g)(\phi)(u)= \phi(g^{-1}u), \qquad \forall u \in U, g \in G, \phi \in L^2(U).$$
The action of $G'$ takes a more complicated form involving a factorization of elements into simpler components, see \cite{Follandbook, Gan-survey}. When we use a concrete model for the oscillator representation, this will be our default one. 
The action of $G'{\times} G$ preserves the space $S(U)$ of Schwartz functions, giving us the {\em smooth} oscillator representation
$$\omega^\infty : G' {\times} G \longrightarrow \mathrm{End}(S(U)).$$

\subsection{$L^2$-duality} Roger Howe proved (see \cite[Theorem  6.1]{Howe-89}) the ``double commutant' result that the actions of $G'$ and $G$ on $L^2(U)$ not only commute with each other but in fact generate each other's commutants in the sense of von Neumann algebras. More precisely, the von Neumann algebras obtained as the closures, in the weak operator topology, of the linear spans of $\omega(G')$ and $\omega(G)$ inside $\mathbb{B}(L^2(U))$ are each other's commutants. This leads to a multiplicity one decomposition of the oscillator representation as a $G' {\times} G$-representation
\begin{equation}
\label{aadadlnadlknda}
\omega \simeq \int_{\widehat{G}} {\pi}' \otimes  \pi \ d\mu(\pi),
\end{equation}
where the integral ranges over irreducible unitary representations $\pi$ of $G$, the ${\pi}'$ are irreducible unitary representations of $G'$ and $\mu$ is some Borel measure on the unitary dual $\widehat{G}$ of $G$. As things are defined up to measure zero sets with respect to $\mu$, the rule $\pi \leftrightarrow {\pi}'$ gives a true correspondence only over the atoms of $\mu$, which of course could be an empty set. The decomposition \eqref{aadadlnadlknda} however sets up a measurable correspondence from the support of $\mu$ to $\widehat{G'}$ by declaring ``$\pi \leftrightarrow {\pi}'$ if and only if ${\pi}' \otimes \pi$ is a subrepresentation of $\omega$''. See \cite[Section 6]{Howe-89}, \cite[Section 2]{Li-90} and \cite[Section 4.6]{Li-00} for more on the $L^2$-duality.

\subsection{Smooth duality} Following Howe, we will move beyond the above unitary setting by considering the smooth oscillator representation. We will also change our perspective from subrepresentations to quotient\footnote{In his unpublished 1996 notes on the local theta correspondence, Steve Kudla writes: ``\emph {Since there are many nontrivial extensions in the category of smooth or even admissible representations of $p$-adic groups, it is best to study quotients rather than summands or submodules.}''} representations. 

Given smooth irreducible representations ${\pi}', \pi$ of $G',G$ respectively, we define the rule that $\pi \leftrightarrow {\pi}'$ if and only ${\pi}' \otimes \pi$ is a quotient of the smooth oscillator representation, that is, if there is a non-trivial surjection from $\omega^\infty$ onto ${\pi}' \otimes \pi$ that is $G'{\times} G$-equivariant. In the archimedean case, we ask for the kernel of this surjection to be closed.  

Howe conjectured that the above rule gives a partial bijection between the smooth duals of $G'$ and $G$. In the archimedean case, Howe proved this conjecture (\cite{Howe-89}). In the non-archimedean case, it is known to be true by works\footnote{We are only listing works relevant to the ortho-symplectic pairs that we are focusing on. If we consider all dual pairs, we should also mention works of Minguez (\cite{Minguez}) and of Gan-Sun (\cite{Gan-Sun}).} of Waldspurger (\cite{Waldspurger-90}) and Gan-Takeda (\cite{Gan-Takeda}).  In the literature the correspondence  $\pi \leftrightarrow {\pi}'$ is usually called {\bf Howe duality}, for obvious reasons, or {\bf theta correspondence}, for its connections with theta lifting in the theory of automorphic forms. We will say that $\pi$ and $\pi'$ are {\em theta lifts} of each other and will use the notation 
${\pi}'=\theta(\pi)$ and ${\pi}=\theta(\pi')$.

We point out that the $L^2$-duality, when restricted to the discrete part of the decomposition, is compatible with the smooth duality. This is discussed in \cite[Section 2]{Li-90} and \cite[Section 4.6]{Li-00}. In the rest of the paper, we will mean smooth duality when we say duality. We also point out that $C^*$-algebras in their own right can only be used to study unitary representations, but when equipping a $C^*$-algebra with a well behaved subalgebra, as is comme il faut in noncommutative geometry, could allow for extending $C^*$-algebraic methods to the setting of smooth duality. 

\subsection{Preservation of unitarity} 
Now that the duality is described as a partial bijection between the smooth duals, it is natural to ask when unitarity is preserved under theta correspondence. It is known that unitarity is not preserved in general; for example, the trivial representation of $O(p,q)$ lifts to a non-unitary representation of ${\rm Sp}_{2n}(\R)$ when $p>n+1$, $pq\not=0$ and $p+q$ is even (see e.g. \cite[p.95]{Lee-Zhu}). There are however important cases in which it is known that unitarity is preserved. We will discuss some of these below. Readers interested in this aspect of theta corrspondence should also see Chen-Bo Zhu's contribution to these proceedings (\cite{Zhu-IHP}).

\section{Compact case of LTC} 
\label{theta-compact-case}
Assume that $G$ is compact\footnote{Thus, $G$ is necessarily the orthogonal group.}. Then all irreducible representations of $G$ are finite dimensional and unitary. It is easy to see that representations of $G$ which enter the theta correspondence lift to unitary representations of $G'$. Thus unitarity is preserved. 

Indeed, the oscillator representation $\omega$ decomposes discretely as a $G$-representation
\begin{equation} 
\label{compact-decomposition} 
\omega \simeq \sum_{\pi \in \widehat{G}} {\rm Hom}_G(\pi, \omega)  \otimes \pi.
\end{equation}
Here ${\rm Hom}_G(\pi, \omega) $ is a Hilbert space since $G$ is compact, and so $\pi$ is finite-dimensional. Indeed, since $\pi$ is finite-dimensional ${\rm Hom}_G(\pi, \omega) $ carries the standard inner-product
$$\langle S,T \rangle := {\rm Tr }(T^*S),$$
where the adjoint is taken with respect to the unitary structures on $\pi$ and $\omega$. As the actions of $G$ and $G'$ commute, $G'$ acts unitarily on the multiplicity spaces ${\rm Hom}_G(\pi, \omega)$ by post-composition and the obtained $G'$-representations are irreducible thanks to double commutation, when the multiplicty spaces are non-zero. Thus, the duality is completely determined within the $L^2$-picture: 
\begin{equation}
\label{thetadaom}
\theta(\pi) = {\rm Hom}_G(\pi, \omega).
\end{equation}

We can identify our multiplicity spaces as spaces of invariants
and using our Schr\"odinger model of the oscillator representation, we can further describe these spaces of invariants more concretely as an $L^2$-space
$${\rm Hom}_G(\pi, \omega) \simeq \left ( \omega \otimes \pi^* \right )^G  \simeq L^2_G(U,V_{\pi^*})$$
where $L^2_G(U,V_{\pi^*})$ is the space of $L^2$-functions on $U$ valued in $V_{\pi^*}$ that are $G$-equivariant.  This identification is $G'$-equivariant and gives a concrete model for $\theta(\pi)$.

We remark that when $F$ is archimedean, Kashiwara and Vergne (\cite{Kashiwara-Vergne}) determined which $\pi$ enter the theta correspondence and proved that the theta lifts $\theta(\pi)$ all have highest weight vectors, which they describe explicitly (see also \cite[Appendix]{Przebinda-96}).  In the non-archimedean case, $G$ can be compact only when the dimension $n$ of $V$ is not more than $4$. The case $n=2$, goes back to work of Shalika and Tanaka \cite{Shalika-Tanaka} and the case $n=4$ is related to the Jacquet-Langlands transfer for $GL_2$. See \cite[Section 5]{Gelbart} or in \cite[Section 4.6]{Prasad} for a discussion.

\subsection{Via coinvariants} \label{prelude-Li-method} As a toy model for the method of Jian-Shu Li, let us describe the above multiplicity spaces not as invariants, but as coinvariants. Fix an irreducible unitary representation $\pi$ of our compact group $G$ and, for convenience, let $\mathcal{Z}$ denote the tensor product Hilbert space $L^2(U)\otimes V_{\pi}$ with the diagonal action of $G$. 

Consider the averaging operator $P$ on $\mathcal{Z}$ defined by
\begin{equation} \label{averaging-projection} P(z):= \int_G g{\cdot}z \ dg,
\end{equation}
where we normalize the volume of $G$ as $1$. The operator $P$ is a projection. Its image is the subspace 
$\mathcal{Z}^G$ of $G$-fixed vectors and $\mathcal{Z}/{\rm Ker}(P)$ is the largest quotient of $\mathcal{Z}$ on which $G$ acts trivially, in other words, the coinvariant space $\mathcal{Z}_G$. Thus, the operator $P$ establishes an isomorphism
\begin{equation} \label{coinv-inv-isom}
(\omega \otimes \pi)_G \simeq (\omega \otimes \pi)^G
\end{equation}

Let $\left ({\cdot},{\cdot}\right )$ denote the inner product on $\mathcal{Z}$ and introduce the form 
\begin{equation} \label{Li-form-1} \llangle z_1,z_2 \rrangle := \left (P(z_1), z_2 \right )
\end{equation}
for $z_1,z_2 \in \mathcal{Z}$. 
As $P$ is self-adjoint, the form $\llangle {\cdot},{\cdot} \rrangle$ is Hermitian. It is also non-negative since
$$\llangle z,z \rrangle =( P(z),z ) = ( P^2(z),z ) = ( P(z),P(z) ) \geq 0.$$
Moreover, as $\left ({\cdot},{\cdot}\right )$ is non-degenerate, the radical 
$ \{ z \in \mathcal{Z} \mid \llangle z,w \rrangle=0 \quad \forall w \in \mathcal{Z} \}$
of $\llangle {\cdot},{\cdot} \rrangle$ equals the kernel of $P$. Therefore the form $\llangle {\cdot},{\cdot} \rrangle$ descends to an inner product on $(\omega \otimes \pi)_G$. As the operator $P$ commutes with the $G'$-action, the form $\llangle {\cdot},{\cdot} \rrangle$ is $G'$-invariant. Moreover, the action of $G'$ descends to $(\omega \otimes \pi)_G$ and the isomorphism \eqref{coinv-inv-isom} is $G'$-equivariant. Therefore the $G'$-representation on 
$(\omega \otimes \pi)_G$ captures $\theta(\pi^*)$.

\section{Method of Li and preservation of unitarity}
\label{sec:liswork}
In \cite{Li-89}, Jian-Shu Li approached theta lifting with a strategy that had important applications to the preservation of unitarity question. As a prelude, we have presented this method in Section \ref{prelude-Li-method} where compactness of the group removed all the technical challenges. We will now discuss the method in full generality. 

Let us fix a dual pair $G',G$ with $G$ the smaller group\footnote{For size, we simply compare the dimensions of $V'$ and $V$. In the case that they are equal, we will say that the orthogonal group is smaller. More conceptually, we can compare the size of the dual groups; the dual of ${\rm Sp}_{2n}$ is ${\rm SO}_{2n+1}$ and the dual of ${\rm O}_{2n}$ is ${\rm SO}_{2n}$. }.  Given an irreducible unitary representation $(\pi,V_\pi)$ of $G$, we consider the algebraic tensor product $S(U)\otimes V_\pi^\infty$ where 
$S(U)$ is the Schwartz subspace of $L^2(U)$ and $V_\pi^\infty$ is the subspace of smooth vectors in $V_\pi$. 

We equip $S(U)\otimes V_\pi^\infty$ with the form
\begin{equation} \label{Li-form-2} (\phi_1 \otimes v_1, \phi_2 \otimes v_2)_\pi :=
\int_G \langle \phi_1,\omega(s)(\phi_2) \rangle \langle  v_1,\pi(s)(v_2)\rangle \mathrm{d}s.
\end{equation}
Of course, this is just our form \eqref{Li-form-1} but in this generality, we do not know whether this integral is convergent. Let us assume that it is. It is easy to see that this form is Hermitian and $G'$-invariant with respect to the natural action $\omega \otimes {\bf 1}$ of $G'$.  Let $\mathcal{R}_{\pi}$ denote the radical of $({\cdot},{\cdot})_\pi$. Then $\mathcal{R}_{\pi}$ is stabilised by $G'$ and thus the quotient space  
$$\left ( \mathbb{S} \otimes V_{\pi}^\infty \right )/\mathcal{R}_{\pi}$$
affords a $G'$-representation that we will denote by $L(\pi)$. Assuming furthermore that \eqref{Li-form-2} is non-negative, we obtain a positive-definite $G'$-invariant form on $L(\pi)$. 

A well-known argument based on the description of the theta lift as the maximal semisimple quotient of  the so-called `big theta lift' leads to the following result. 
\begin{proposition} \label{Li=theta} Assume that the integral in \eqref{Li-form-2} is convergent and the form it defines is non-negative. If the unitary $G'$-representation $L(\pi)$ is nonzero, then it is isomorphic to $\theta(\pi^*)$. 
\end{proposition}

\subsection{Stable range case} \label{stable-range-case}
Another important case in which unitary is preserved is the stable range case. This is the case where $V$ has a  maximal isotropic subspace whose dimension is greater than or equal to the dimension of $V'$, or vice versa. Thus, if $F$ is real, then the stable range pairs are $({\rm Sp}_{2n}(\R),O(p,q))$ with $n\geq p+q$ or $2n \leq {\rm min}\{p,q\}$. In the stable range setting, Jian-Shu Li proved in \cite{Li-89} that every unitary representation of the smaller group enters the theta correspondence and that their theta lifts are unitary as well. 

\begin{theorem}
\label{Li-unitary}
(Li \cite{Li-89})  Let $(G',G)$ be a stable range dual pair with $G$ the smaller member. Assume that $(G',G)$ is not the pair $(O_{2n,2n},Sp_{2n})$. Let $\pi$ be an irreducible unitary representation of $G$. Then
\begin{enumerate} 
\item the integral $\eqref{Li-form-2}$  is convergent and the form $({\cdot}, {\cdot})_\pi$ is non-negative,
\item $L(\pi)$ is nonzero.
\end{enumerate}
Therefore, by Prop. \ref{Li=theta}, all of the unitary dual of $G$ enters the theta correspondence and their lifts are also unitary.
\end{theorem}

\subsection{Tempered case} \label{tempered-case}
More generally, Li's method can be used to show that if a tempered unitary representation of the smaller group enters the theta correspondence, then it lifts to a unitary representation.

\begin{theorem} \label{tempered-reps} Let $(G',G)$ be a dual pair with $G$ the smaller member. Let $\pi$ be a tempered irreducible representation of $G$. Then 
\begin{enumerate}
\item the integral $\eqref{Li-form-2}$  is convergent and the form $({\cdot}, {\cdot})_\pi$ is non-negative,
\item $L(\pi)$ is non-zero if and only if $\pi$ enters the theta correspondence.
\end{enumerate} 
Therefore, by \eqref{Li=theta}, the tempered representations of $G$ entering the theta correspondence lift to unitary representations.
\end{theorem}

This is well-known among specialists. We refer to \cite[Thm. 3.4]{Goffeng-Mesland-Sengun} for a proof. We note that for the special case of discrete series, the above result is proven in \cite[Proposition 2.4]{Li-90}.

\section{Compact case of LTC via group $C^*$-algebras}
\label{compaclatatd}
Let us now revisit the compact case theta correspondence discussed in Section \ref{theta-compact-case} from the perspective of group $C^*$-algebras and Hilbert $C^*$-modules. In a nutshell, we will build a $(C^*(G'),C^*(G))$-correspondence from the oscillator representation whose associated induction functor will ``capture'' theta correspondence in a way which could be seen as a `group algebra' version of the approach in Section \ref{prelude-Li-method}.

Recall from Section \ref{theta-background} that we can realize the oscillator representation on some Hilbert space of the form $L^2(U)$ with the smooth oscillator representation captured on the subspace $S(U)$ of Schwartz functions on $U$.

\subsection{Hilbert module over $C^*(G)$: first method}  We will build our $C^*(G)$ by defining modules over the local components of $C^*(G)$ and then assembling those local modules. Let $L^2(U)_\pi$ denote the $\pi$-isotypical component of subspace of $L^2(U)$ so that we have
$$L^2(U) \simeq \bigoplus_{\pi \in \widehat{G}} L^2(U)_{\pi^*}.$$
Here $\pi^*$ denotes the contragradient representation of $\pi$, i.e. $\pi^*$ acts on $V_\pi^*$ via the dual action. It will soon become apparent why we use the isotopical components with respect to contragradient representations, but in short it relates to Example \ref{first-example}. Each isotypical component $L^2(U)_{\pi^*}$  decomposes as
$$L^2(U)_{\pi^*} \simeq {\rm Hom}_G(V_\pi^*,L^2(U)) \otimes V_\pi^*,$$
with $G$ acting as ${\bf 1}\otimes \pi^*$ on the right hand side. As we discussed in Section  \ref{first-example}, $V_\pi^*$ has the structure of a right Hilbert module over $\End(V_\pi)$ and ${\rm Hom}_G(V_\pi^*,L^2(U))$ is a Hilbert space. This allows us to view the tensor product ${\rm Hom}_G(V_\pi^*,L^2(U)) \otimes V_\pi^*$ as an external Hilbert module tensor product, hence $L^2(U)_{\pi^*}$ is a Hilbert module over $\End(V_\pi)$ in the obvious way. We can view $\End(V_\pi)\cong J_\pi$ as an ideal in $C^*(G)$, so $L^2(U)_{\pi^*}$ forms a  Hilbert $C^*$-module over $C^*(G)$.  We use the notation $X_\pi$ for the right Hilbert $C^*$-module $L^2(U)_{\pi^*}$ over $C^*(G)$. Bundling these modules together, we define the right $C^*(G)$-Hilbert $C^*$-module
$$X:=\bigoplus_{\pi \in \widehat{G}} X_\pi$$
where the right hand side is the \emph{$\ell^2$-direct sum} of Hilbert $C^*$-modules
$$\left \{ (\phi_\pi)_{\pi} \in \prod_{\pi\in \widehat{G}}  X_\pi \mid \sum_{\pi\in \widehat{G}}  \langle\phi_\pi,\phi_\pi\rangle_{X_\pi} < \infty \right \}.$$
We note that since $G$ is compact, there is a dense $C^*(G)$-linear inclusion $L^2(U)\hookrightarrow X$ as a proper subspace.

\subsection{Hilbert module over $C^*(G)$: second method} \label{hilbert-module-compact-group-2} Alternatively, we can get our Hilbert module over $C^*(G)$ in a direct manner, without appealing to the decomposition of $C^*(G)$. It takes more effort but it has the advantage of being applicable to more general settings as we will see later.

Let $C^\infty(G)$ be the convolution algebra of smooth functions on $G$, viewed as a dense subalgebra of $C^*(G)$. 
We view $S(U)$ as a right $C^\infty(G)$-module:
$$\phi{\cdot}a := \int_G a(s) \omega(s^{-1})(\phi) ds, \quad \phi \in S(U), \ a \in C^\infty(G).$$
As in \eqref{left-inner-compact-group}, we form the matrix coefficient map $S(U) \times S(U) \to C^\infty(G)$ given by
$$\langle \phi_1, \phi_2 \rangle\subrangle{G}(s) := \langle \phi_1, \omega(s)(\phi_2) \rangle\subrangle{L^2(U)}.$$
This gives us a right linear, Hermitian $C^*(G)$-valued form on $S(U)$ which is compatible with the right action of $C^\infty(G)$. 

Next, we need to argue that given $\phi \in S(U)$, the element $\langle \phi, \phi \rangle\subrangle{G}\in C^*(G)$ is positive. Equivalently, for any unitary irreducible representation $(\pi,V_\pi)$ of $G$, we need to show that $\pi(\langle \phi, \phi \rangle\subrangle{G})$ is a positive operator on $V_\pi$. Given $v \in V_\pi$, we have
\begin{align} \label{localisation-IP-averaging} \langle v, \pi(\langle \phi, \phi \rangle \subrangle{G})(v) \rangle \subrangle{V_\pi} &=
\int_G \langle \phi, \omega(s)(\phi)\rangle \subrangle{L^2(U)} \langle v,\pi(s)(v) \rangle\subrangle{V_\pi}   \ ds \\
&= \int_G \langle \phi \otimes v, (\omega \otimes \pi)(s)(\phi \otimes v) \rangle\subrangle{L^2(U)\otimes V_\pi}  \ ds \\
&=  \langle \phi \otimes v, P(\phi \otimes v) \rangle\subrangle{L^2(U)\otimes V_\pi}  
\end{align}
where $P$ is the averaging operator
$$P(\phi \otimes v):=\int_G (\omega \otimes \pi)(s)(\phi \otimes v) \ ds $$
on the Hilbert space $L^2(U)\otimes V_\pi$ that we encountered earlier in \eqref{averaging-projection}. As discussed in Section \ref{prelude-Li-method}, $P$ is a projection and hence 
$$ \langle \phi \otimes v, P(\phi \otimes v) \rangle\subrangle{L^2(U)\otimes V_\pi}  \geq 0,$$
proving that $\langle \phi, \phi \rangle\subrangle{G}$ is a positive element. It follows quickly along the same lines that 
$\langle \phi, \phi \rangle\subrangle{G}$ is 0 if and only if $\phi$ is 0. 

Standard arguments now tell us that the completion $X_S$ of $S(U)$ with respect to the norm arising from the inner product 
$\langle {\cdot},{\cdot} \rangle\subrangle{G}$ acquires a Hilbert module structure over $C^*(G)$ extending the $C^\infty(G)$-module structure and the inner product on $S(U)$.

\subsection{More on the the $C^*(G)$-Hilbert modules $X$ and $X_S$} Put $S(U)_{\pi}$ for $S(U) \cap L^2(U)_{\pi^*}$. The matrix coefficients $\langle S(U)_\pi, S(U)_\pi \rangle\subrangle{G}$ land in the dense subalgebra $J_\pi \cap C^\infty(G)$ of $J_\pi$ and $S(U)_\pi$ is a module over  $J_\pi \cap C^\infty(G)$ as well. Thus we can complete $S(U)_\pi$ to a Hilbert module $X_{S,\pi}$ over $J_\pi$. The decomposition of $C^*(G)$ given in \eqref{decomp-compact-Cstar} leads to the decomposition of $C^*(G)$-Hilbert $C^*$-modules.
$$X_S \simeq \bigoplus_{\pi \in \widehat{G}} X_{S,\pi}.$$
A short approximation argument shows that $S(U)_\pi\subseteq L^2(U)_{\pi^*}$ is dense, and therefore we have that 
$$X\cong X_S,$$ 
as right $C^*(G)$-Hilbert $C^*$-modules. In fact, we have that
\begin{align}
\label{localsinadinacom}
X\otimes_{C^*(G)}V_{\pi} & = X\otimes_{J_\pi}V_{\pi}  =X_\pi\otimes_{J_\pi}V_{\pi}\\
\nonumber
=&{\rm Hom}_G(V_\pi^*,L^2(U)) \otimes V_\pi^*\otimes_{\End(V_\pi)}V_{\pi}={\rm Hom}_G(V_\pi^*,L^2(U)).
\end{align}
Here we use that $V_\pi^*\otimes_{\End(V_\pi)}V_{\pi}\cong \C$. 

\subsection{Promotion to a $(C^*(G'),C^*(G))$-correspondence} The action of $G'$ on $S(U)$ gives rise to unitary Hilbert $C^*(G)$-module operators on $X$. Standard arguments then show that the action of $G'$ on $S(U)$, just like in the Hilbert space case, `integrates' to a representation
$$C^*(G') \to \End^*_{C^*(G)}(X),$$
turning $X$ into a $(C^*(G'),C^*(G))$-correspondence. Alternatively, we know that $G'$ acts on $X_\pi={\rm Hom}_G(V_\pi,L^2(U)) \otimes V_\pi$ via the first factor, making $X_\pi$ a $\left ( C^*(G'),\End(V_\pi) \right )$-correspondence. The $\ell^2$-direct sum of these correspondences makes $X$ a $\left (C^*(G'),C^*(G)\right )$-correspondence. In any case, the key feature of this correspondence is that the associated induction functor is, not surprisingly, is directly related to theta correspondence.

\begin{theorem} 
Let $G$ be compact and $X$ be the $(C^*(G'),C^*(G))$-correspondence built from the oscillator representation as above. Let $\pi$ be an irreducible unitary representation of $G$. Then we have 
\begin{equation} 
\label{theta-induction} 
{\rm Ind}_{C^*(G)}^{C^*(G')}(X)(\pi)\simeq \theta(\pi^*).
\end{equation}
\end{theorem}

Here we make the convention that if $\sigma$ does not enter the theta correspondence then $\theta(\sigma)=0$.

\begin{proof} By \eqref{localsinadinacom}, we have
$${\rm Ind}_{C^*(G)}^{C^*(G')}(X)(\pi) = X \otimes_{C^*(G)} V_\pi= {\rm Hom}_G(V_\pi^*,L^2(U)).$$
But ${\rm Hom}_G(V_\pi^*,L^2(U)) \simeq \theta(\pi^*)$ by \eqref{thetadaom}.
\end{proof}

\section{Method of Li via group $C^*$-algebras} 
\label{bimodule-construction}
Following \cite{Goffeng-Mesland-Sengun}, we will now discuss how the approach of Section \ref{hilbert-module-compact-group-2} generalizes beyond $G$ compact to the two cases of theta correspondence in which Li's method applies, namely, those described in Section \ref{stable-range-case} and Section \ref{tempered-case}. Let $(G',G)$ be a dual pair with $G$ the smaller group. Let $\omega$ be the oscillator representation of $G'\times G$ afforded on some $L^2(U)$ with $S(U)$ the subspace of Schwartz functions. Let $C_c^\infty(G)$ be the convolution algebra of smooth compactly-supported functions on $G$. We view $S(U)$ as a right $C_c^\infty(G)$-module:
$$\phi{\cdot}a := \int_G a(s) \omega(s^{-1})(\phi) ds, \quad \phi \in S(U), \ a \in C_c^\infty(G).$$
We form the matrix coefficient map on $S(U)$:
$$\langle \phi_1, \phi_2 \rangle\subrangle{G}(s) := \langle \phi_1, \omega(s)(\phi_2) \rangle\subrangle{L^2(U)}.$$
It follows from well-known matrix coefficients estimates of the oscillator representation that these have very fast decay (as $G$ is the smaller group), ensuring that they lie in the reduced group $C^*$-algebra $C^*_r(G)$. In fact, if $(G',G)$ is in the stable range, then we can conclude that these matrix coefficients are integrable and hence, they lie in $C^*(G)$. This gives us a right linear, Hermitian $C_r^*(G)$-valued ($C^*(G)$-valued if we are in the stable range case) form on $S(U)$ which is compatible with the right action of $C_c^\infty(G)$. 

Next, we need to argue that given $\phi \in S(U)$, the element $\langle \phi, \phi \rangle\subrangle{G}$ of $C_r^*(G)$ is positive. 
We proceed as in Section \ref{hilbert-module-compact-group-2}.  Given an irreducible tempered unitary representation $(\pi,V_\pi)$ of $G$, we want to prove that for every $v \in V_\pi$ and $\phi \in S(U)$, we have
\begin{equation} \label{positivity} \langle v, \pi(\langle \phi, \phi\rangle\subrangle{G})(v)\rangle\subrangle{V_\pi} \geq 0.
\end{equation}
The key observation, which we made a couple of times already, is that this is the same as asking for positivity of Li's inner product $({\cdot},{\cdot})_\pi$ defined in \eqref{Li-form-2} since
\begin{align*} \langle v, \pi(\langle \phi, \phi\rangle\subrangle{G})(v)\rangle\subrangle{V_\pi} &=
\int_G \langle \phi, \omega(s)(\phi)\rangle \subrangle{L^2(U)} \langle v,\pi(s)(v) \rangle\subrangle{V_\pi}   \ ds \\
&= ( \phi \otimes v, \phi \otimes v )_\pi.
\end{align*}
Therefore, \eqref{positivity} is the content of Theorem \ref{tempered-reps}. In the stable range case, we ask for \eqref{positivity} for all irreducible unitary representations $\pi$ of $G$ and it is given by Theorem \ref{Li-unitary}.

Standard arguments now tell us that the completion, denoted $\Theta$, of $S(U)$ with respect to the norm arising from the inner product $\langle {\cdot},{\cdot} \rangle\subrangle{G}$ acquires a Hilbert module structure over $C_r^*(G)$ (in the stable range case, over $C^*(G)$) extending the $C_c^\infty(G)$-module structure and the inner product on $S(U)$. The action of $G'$ on $S(U)$ gives rise to unitary Hilbert $C_r^*(G)$-module operators on $\Theta$. Again, standard arguments show that the action of $G'$ on $S(U)$ integrates to an action of  $C^*(G')$ on $\Theta$ by adjointable Hilbert module operators, turning $\Theta$ into a $(C^*(G'),C^*_r(G))$-correspondence (in the stable range, a $(C^*(G'),C^*(G))$-correspondence), which we call the {\bf theta bimodule}.

Again, the punchline is that the induction functor associated to $\Theta$ captures the theta correspondence.
\begin{theorem}Let $(G',G)$ be a dual pair with $G$ the smaller group and $\Theta$ be the theta bimodule defined above. Let $\pi$ be an irreducible tempered unitary representation of $G$. Then we have 
\begin{equation} \label{theta-induction-2} {\rm Ind}_{C^*_r(G)}^{C^*(G')}(\Theta)(\pi)\simeq \theta(\pi^*).
\end{equation}
If $(G',G)$ is in the stable range case then for any irreducible unitary representation $\pi$ of $G$, we have
\begin{equation} \label{theta-induction-3} {\rm Ind}_{C^*(G)}^{C^*(G')}(\Theta)(\pi)\simeq \theta(\pi^*).
\end{equation}
\end{theorem}

\section{Equal rank LTC as a Morita equivalence}
\label{seceualada}

In this section, following \cite{Mesland-Sengun}, we will consider ortho-symplectic dual pairs $({\rm Sp}_{2n},{\rm O}_{2n+1})$. These are called {\em equal rank} pairs because their dual groups are the same; namely ${\rm SO}_{2n}$. In this special case, both groups enjoy being the smaller group and as a result, the theta correspondence preserves many important features. For example, temperedness is preserved and discrete series representations lift to discrete series representations, with their formal dimensions preserved.

You will recall that at the beginning of the paper, we considered only even ortho-symplectic papers. The reason was only technical; for the dual pair $({\rm Sp}_{2n},{\rm O}_{2n+1})$; the oscillator representation 
$$\omega: {\rm Mp}_{2n} \times {\rm O}_{2n+1} \longrightarrow \mathcal{U}(L^2(U))$$
requires passing to the metaplectic double cover ${\rm Mp}_{2n}$ of the symplectic group. Once this technical point is established, our $C^*$-algebraic approach works just the same.

Let $(G',G)$ denote the equal rank dual pair above with the convention that we replace the symplectic group with its metaplectic double cover. Exploiting the fact that matrix coefficients of the oscillator representation have very fast decay for both groups (since both groups are now the smaller group), we can equip $S(U)$ not just with a right Hilbert module structure over $C^*_r(G)$ as we did in Section \ref{bimodule-construction}, but also with a left Hilbert module structure over $C^*_r(G')$. 

The key observation is the following. 
\begin{theorem} For an equal rank dual pair $(G',G)$, the theta bimodule $\Theta$ that we constructed in Section \ref{bimodule-construction} is actually an imprimitivity bimodule. 
\end{theorem}

In other words, the two $C^*$-valued inner products $\prescript{}{G'}\langle {\cdot},  {\cdot} \rangle$ and 
$\langle  {\cdot}, {\cdot} \rangle\subrangle{G}$ on $\Theta$ satisfy the comptability property \eqref{imprimitivity-compatibility}: 
$$ \prescript{}{G'}\langle x,y \rangle{\cdot} z= x{\cdot}\langle y, z \rangle\subrangle{G}$$
for all $x,y,z \in \Theta$. This is proven in \cite[Section 5.4]{Mesland-Sengun} and it follows from Gan and Ichino's local version of the Rallis inner product formula in \cite{Gan-Ichino}.

Let $C^*_{\theta}(G')$ and $C^*_{\theta}(G)$ be the two-sided closed ideals of $C^*_r(G')$ and $C^*_r(G)$, respectively, cut out by the $C^*$-valued inner products $\prescript{}{G'}\langle {\cdot},  {\cdot} \rangle$ and $\langle  {\cdot}, {\cdot} \rangle\subrangle{G}$ on $\Theta$  (as defined in Section \ref{Hilbert-module}). As $\Theta$ is an imprimitivity bimodule, it establishes a Morita equivalence between $C^*_{\theta}(G')$ and $C^*_{\theta}(G)$. This Morita equivalence explains much of the attractive features that equal rank local theta correspondence enjoys. 

\subsection{} To illustrate this, let us start by recalling that $C^*_{\theta}(G')$ and $C^*_{\theta}(G)$ are complemented ideals. This implies that their duals $\widehat{C^*_{\theta}(G')}$ and 
$\widehat{C^*_{\theta}(G)}$ are clopen subsets of the reduced duals of $G'$ and $G$. Morita equivalence implies that $\widehat{C^*_{\theta}(G')}$ and $\widehat{C^*_{\theta}(G)}$ are homeomorphic. It can be shown that the dual $\widehat{C^*_{\theta}(G')}$ and $\widehat{C^*_{\theta}(G)}$ are precisely the subset of the reduced duals of $G'$ and $G$ that enter the theta correspondence. In particular, if we take a discrete series representation $\pi$ of $G$ so that $\pi$ is isolated in the reduced dual of $G$, its image 
${\rm Ind}(\Theta)(\pi) \} \simeq \{\theta(\pi^*)$ will be isolated in $\widehat{C^*_{\theta}(G')}$. As the latter is clopen, we conclude that $\theta(\pi^*)$ is isolated in the reduced dual of $G'$. Thus $\theta(\pi^*)$, and hence $\theta(\pi)$, is a discrete series representation. A more sophisticated discussion (see \cite[Section 8]{Mesland-Sengun}) that involves $K$-theory and canonical traces of $C^*$-algebras shows that the formal dimensions of discrete series are also preserved, explaining this well-known fact within our Morita equivalence framework.

\subsection{} As another application of this Morita equivalence, we close our discussion with a transfer of characters result, originally due to Gan (\cite{Gan-20, Gan-21}).   Recall that if $\pi$ is a tempered irreducible representation of a reductive algebraic group $H$ over a local field, the character of $\pi$ is the tempered distribution on $H$, that is, the continuous linear functional 
$${\rm ch}(\pi) : \mathcal{S}(H) \to \C$$
on Harish-Chandra's Schwartz algebra $\mathcal{S}(H)$ of $H$ given by the trace 
$${\rm ch}(\pi)(\varphi):= {\rm tr} \ \pi(\varphi).$$
In our setting, the oscillator bimodule forms a connection between parts of the Schwartz algebras of the equal rank dual pair groups $G'$ and $G$, and as such gives a meaningful way of expressing the character one representation in terms of that of its theta lift. 

Indeed, given $x,y \in S(U) \subset \Theta$, recall that the matrix coefficient functions $\prescript{}{G'}\langle x,y \rangle \in C^*_r(G')$ and
$\langle x, y \rangle\subrangle{G} \in C^*_r(G)$ belong to the Harish-Chandra Schwartz algebras of the respective groups. As mentioned above, this fast decay property ensures that that for any irreducible tempered representations $\pi'$ of $G'$ and $\pi$ of $G$, the operators $\pi'(\langle x, y \rangle\subrangle{G})$ and $\pi(\langle x, y \rangle\subrangle{G})$ are of trace class.  The following is an unpublished result of Gan that is proven in \cite[Section 7]{Mesland-Sengun} as a short elementary consequence of the above Morita equivalence.
\begin{theorem}  Let $(G',G)$ be an equal rank dual pair. Let $\pi$ be a tempered irreducible representation of $G$ that enters the theta correspondence. Given $x,y \in S(U)$, we have
 $${\rm ch}(\theta(\pi))(\prescript{}{G'}{\langle x,y \rangle}) = {\rm ch}(\pi)(\langle y,x \rangle\subrangle{G}).$$
\end{theorem}


\end{document}